\documentclass[12pt,reqno]{amsart}

\usepackage{multirow}
\usepackage{hhline}

\usepackage{mathtools}
\usepackage{times}
\usepackage[T1]{fontenc}
\usepackage{mathrsfs}
\usepackage{latexsym}
\usepackage[titletoc, title]{appendix}
\usepackage{amsmath,amsfonts,amsthm,amssymb,amscd}
\usepackage[dvipsnames]{xcolor}
\usepackage[hidelinks]{hyperref}
\usepackage{amsmath}
\usepackage[utf8]{inputenc}

\usepackage{breakurl}

\newcommand{\bburl}[1]{\textcolor{blue}{\url{#1}}}

\newtheorem{thm}{Theorem}[section]

\newtheorem{cor}[thm]{Corollary}
\newtheorem{claim}[thm]{Claim}
\newtheorem{lem}[thm]{Lemma}
\newtheorem{prop}[thm]{Proposition}

\newtheorem{defi}[thm]{Definition}
\newtheorem{rek}[thm]{Remark}

\usepackage[utf8]{inputenc}

\DeclareFixedFont{\ttb}{T1}{txtt}{bx}{n}{12} 
\DeclareFixedFont{\ttm}{T1}{txtt}{m}{n}{12}  

\usepackage{color}
\definecolor{deepblue}{rgb}{0,0,0.5}
\definecolor{deepred}{rgb}{0.6,0,0}
\definecolor{deepgreen}{rgb}{0,0.5,0}

\usepackage{listings}

\newcommand\pythonstyle{\lstset{
language=Python,
basicstyle=\ttm,
morekeywords={self},              
keywordstyle=\ttb\color{deepblue},
emph={MyClass,__init__},          
emphstyle=\ttb\color{deepred},    
stringstyle=\color{deepgreen},
frame=tb,                         
showstringspaces=false
}}

\lstnewenvironment{python}[1][]
{
\pythonstyle
\lstset{#1}
}
{}

\newcommand\pythoninline[1]{{\pythonstyle\lstinline!#1!}}

\newcommand{\phiAGone}{\ensuremath{\phi\text{-AG}_1}}
\newcommand{\phiSLCone}{\ensuremath{\phi\text{-SLC}_1}}
\newcommand{\weakphiSLCone}{\ensuremath{\text{weak}\ \phi\text{-SLC}_1}}
\newcommand{\weightedweakphiSLCone}{\ensuremath{\text{weighted\ weak}\ \phi\text{-SLC}_1}}
\newcommand{\Fphione}{\ensuremath{(F_\phi)_1}}

\DeclareMathOperator{\supp}{supp}

\DeclareMathOperator{\sgn}{sgn}
\numberwithin{equation}{section}

\DeclareFontFamily{U}{mathx}{}
\DeclareFontShape{U}{mathx}{m}{n}{<-> mathx10}{}
\DeclareSymbolFont{mathx}{U}{mathx}{m}{n}
\DeclareMathAccent{\widehat}{0}{mathx}{"70}
\DeclareMathAccent{\widecheck}{0}{mathx}{"71}

\begin{document}


\title{Isometric theory for greedy-type properties with enlarged greedy sums}

\author{Miguel Berasategui}
\address{UBA - Pab I, Facultad de Ciencias Exactas y
Naturales, Universidad de Buenos Aires, (1428), Buenos Aires, Argentina}
\email{mberasategui@dm.uba.ar}

\author{H\`ung Vi\d{\^e}t Chu}
\address{Department of Mathematics, Washington and Lee University, Lexington, VA 24450, USA}
\email{hchu@wlu.edu}

\subjclass[2020]{41A65, 46B15}

\keywords{bases, greedy algorithm, isometric theory, greedy sums}

\thanks{The second author gratefully acknowledges support from the Washington and Lee University's 2026 Summer Lenfest Grant.}

\begin{abstract}
For a semi-normalized Markushevich basis, we study isometric almost-greedy inequalities in which a greedy
sum of $\phi(m)$ terms is compared with coordinate projections onto at
most $m$ terms. Here $\phi$ is any strictly increasing map of the
nonnegative integers into themselves with $\phi(0)=0$. We sandwich $\phi$-almost greediness with constant $1$ ($\phi$-AG$_1$) between two levels of symmetry for largest coefficients (SLC):
$$\phi\mbox{-SLC}_1\ \Longrightarrow\ \phi\mbox{-AG}_1\ \Longrightarrow\ \mbox{weak }\phi\mbox{-SLC}_1.$$ Both implications are equivalences if and only if 
$\phi(1)\le2$. We then characterize $\phi$-AG$_1$ by the weighted weak $\phi$-SLC or a single homogeneous norm inequality, called Property $(F_\phi)$.
\end{abstract}

\maketitle

\tableofcontents

\section{Introduction}
Let $(X, \|\cdot\|)$ be a separable Banach space over
$\mathbb{F}\in\{\mathbb{R},\mathbb{C}\}$. Throughout the paper we work with a
\emph{semi-normalized Markushevich basis}
$\mathcal{B} =(e_n,e_n^*)_{n=1}^\infty$ for $X$: this is a biorthogonal
system with $e_n\in X$, $e_n^*\in X^*$, such that
\begin{equation}\label{MBasisAssumptions}
 e_n^*(e_j)\ =\ \delta_{nj},\qquad
 \overline{\operatorname{span}}\{e_n:n\in\mathbb N\}\ =\ X,\qquad
 \bigcap_{n=1}^\infty\ker(e_n^*)\ =\ \{0\},
\end{equation}
and
\begin{equation}\label{MBoundedAssumptions}
 0\ <\ \inf_n\|e_n\|\ \le\ \sup_n\|e_n\|\ <\ \infty,
 \qquad\sup_n\|e_n^*\|\ <\ \infty.
\end{equation}
In particular, the coordinate functionals are uniformly bounded; this is the
$M$-bounded setting in greedy approximation (see \cite{AABW}). 

For each $x\in X$, the coefficient sequence $(e_n^*(x))_n$ belongs to
$c_0$. Indeed, this is clear for finitely supported vectors and follows
for general $x$ from density of their span and
$\sup_n\|e_n^*\|<\infty$. Consequently greedy sets of every finite order
exist. All projections $P_A(x)=\sum_{n\in A}e_n^*(x)e_n$ appearing below
are taken over finite sets $A$.

To approximate $x$ by a linear combination of $m$ basis vectors, the Thresholding Greedy Algorithm sums up $m$ largest coefficients, namely 
$$G_m(x)\ :=\ \sum_{n\in A} e_n^*(x)e_n,\mbox{ where }|A| = m\mbox{ and }\min_{n\in A} |e_n^*(x)| \ \ge\ \max_{n\notin A} |e_n^*(x)|.$$
Such a set $A$ is called a \textit{greedy set of $x$ of order $m$} and is not necessarily unique. Let $\mathcal{G}(x,m)$ denote the collection of all greedy sets of $x$ of order $m$.  

For general biorthogonal systems, greedy approximants need not converge to
$x$; even for Schauder bases there are examples of this phenomenon
\cite[Example 10.2.2]{AK}. This motivates the following notion.

\begin{defi}\normalfont
A basis is said to be \textit{quasi-greedy} if $\lim_{m\rightarrow\infty} \|G_m(x)-x\| = 0$ for every $x\in X$ and for every sequence $(G_m(x))_{m=1}^\infty$. 
\end{defi}
According to \cite[Theorem 1]{W}, quasi-greedy bases can be characterized by the existence of a constant $C\ge 1$ such that
\begin{equation}\label{e10}\|G_m(x)\|\ \le\ C\|x\|, \mbox{ for all } x\in X, m\in \mathbb{N}, \mbox{ and } G_m(x)\in \mathcal{G}(x,m).\end{equation}
If \eqref{e10} holds, we say that our basis is \textit{quasi-greedy with constant $C$}. If instead, 
$$\|x-G_m(x)\|\ \le\ C\|x\|$$ 
holds, then we say our basis is \textit{suppression quasi-greedy with constant $C$}. Being suppression quasi-greedy is the same as being quasi-greedy with possibly different constants. 

Let us observe that the greedy map $G_m: X\rightarrow X$ is not linear; surprisingly, in the case of a quasi-greedy basis with constant $1$, Albiac and Ansorena \cite{AA0} were able to connect the quasi-greedy property to unconditionality, which is a linear property and is generally stronger than the quasi-greedy property. Here a basis is said to be \textit{suppression unconditional with constant $C$} if 
$$\|x-P_A(x)\|\ \le\ C\|x\|$$
for every $x\in X$ and every finite $A\subset\mathbb{N}$,
where $P_A(x):= \sum_{n\in A}e_n^*(x)e_n$. 

\begin{thm}\cite[Theorem 2.1]{AA0}\label{1QGchar}
A basis is quasi-greedy with constant $1$ if and only if it is suppression unconditional with constant $1$. 
\end{thm}

Studies of greedy-like conditions with constant $1$ have attracted considerable attention over the last two decades \cite{AA0, AA, AW, AABB, BBL}.
Oikhberg \cite{O} studied the quasi-greedy property with gaps, namely the following inequality
\begin{equation}\label{e11}\|G_{m}(x)\|\ \le\ C\|x\|, \mbox{ for all } x\in X, m\in \mathbf m, \mbox{ and } G_m(x)\in \mathcal{G}(x, m),\end{equation}
where $\mathbf m$ is a subsequence of the natural numbers. When $\mathbf m = \mathbb{N}$, we recover \eqref{e10}.
Oikhberg showed that if $\mathbf m$ increases fast enough, then \eqref{e11} may hold, while \eqref{e10} fails. However, as long as $C = 1$, \eqref{e11} is equivalent to suppression unconditionality with constant $1$ (\cite[Proposition 2.7]{O}). This slightly extends Theorem \ref{1QGchar}.

Let us rephrase Oikhberg's result in terms of enlarging greedy sums.
The admissible enlargement functions have a particularly simple description.

\begin{defi}\label{defPhi}\normalfont
Let $\Phi$ denote the family of all strictly increasing functions
$\phi:\mathbb Z_{\ge0}\to\mathbb Z_{\ge0}$ with $\phi(0)=0$.
\end{defi}

Since $\phi$ takes integer values, every $\phi\in\Phi$ satisfies
$\phi(m)\ge m$ for $m\ge0$. The following two alternative descriptions
are useful in our arguments.

\begin{prop}\label{propPhiEquivalent}
Let $\phi:\mathbb Z_{\ge0}\to\mathbb Z_{\ge0}$. The following
conditions are equivalent:
\begin{enumerate}
\item[(i)] $\phi\in\Phi$.
\item[(ii)] $\phi(0)=0$, $\phi(m)\ge m$ for every $m\ge1$, and for all
$m,n,k\in\mathbb Z_{\ge0}$ with $n\ge k$,
\begin{equation}\label{e20}
\phi(n)\ \le\ m\quad\Longrightarrow\quad\phi(n-k)\ \le\ m-k.
\end{equation}
\item[(iii)] There is a nondecreasing function
$\gamma:\mathbb Z_{\ge0}\to\mathbb Z_{\ge0}$ with $\gamma(0)=0$ and
$\phi(m)=m+\gamma(m)$ for every $m\ge0$.
\end{enumerate}
\end{prop}

The proof is deferred to Appendix~\ref{phi_equiv}. In particular,
$\phi(m)-m$ is nondecreasing, a fact used repeatedly below.

We restate \cite[Proposition 2.7]{O} as follows. 
\begin{prop} \cite[Proposition 2.7]{O}\label{Tprop} Let $\phi\in \Phi$.
A basis satisfies
$$\|G_{\phi(m)}(x)\|\ \le\ \|x\|, \mbox{ for all }x\in X, m\in \mathbb{N}, \mbox{ and }G_{\phi(m)}(x)\in \mathcal{G}(x,\phi(m))$$
if and only if it is suppression unconditional with constant $1$. 
\end{prop}
Hence, enlarging greedy sums in the condition defining the quasi-greedy property with constant $1$ does not change the property. Our first main goal is to see the effect of enlarging greedy sums in the condition defining the \textit{almost greedy} property with constant $1$. We recall the definition of almost greedy bases.

\begin{defi}\normalfont
A basis is said to be \textit{almost greedy} with constant $C$ if 
$$\|x-G_m(x)\|\ \le\ C\inf_{|A|\le m}\|x-P_A(x)\|, \mbox{ for all } x\in X, m\in \mathbb{N}, \mbox{ and } G_m(x)\in \mathcal{G}(x,m).$$
\end{defi}
Albiac and Ansorena \cite{AA} characterized almost greedy bases with constant $1$ by the so-called \textit{Property (A)}. We need some notation to state Property (A). Let $\mathbb{N}^{<\infty}$ denote the collection of all finite subsets of $\mathbb{N}$, and for $x\in X$, let $\supp(x) = \{n: e_n^*(x)\neq 0\}$. For $A\in \mathbb{N}^{<\infty}$, we write $1_{A}$ to mean the vector $\sum_{n\in A}e_n$. A sign $\varepsilon = (\varepsilon_n)_{n=1}^\infty$ is a sequence of scalars of modulus $1$. Given a sign $\varepsilon$ and a set $A\in \mathbb{N}^{<\infty}$, we write $1_{\varepsilon A}$ to mean $\sum_{n\in A}\varepsilon_n e_n$. Finally, $\|x\|_\infty:= \max_{n\in \mathbb{N}}|e_n^*(x)|$.

\begin{defi}\normalfont
A basis is said to have \textit{Property (A)} if and only if 
$$\|x+1_{\varepsilon A}\|\ \le\ \|x + 1_{\delta B}\|,$$
for all $x\in X$ with $\|x\|_\infty\le 1$, for all $A, B\in \mathbb{N}^{<\infty}$ with $|A|\le |B|$ and pairwise disjoint $A, B, \supp(x)$, and for all signs $\varepsilon, \delta$. 
\end{defi}

\begin{thm}\cite[Theorem 2.3]{AA}\label{AGchar}
A basis is $1$-almost greedy if and only if it has Property (A). 
\end{thm}

Inspired by Theorem \ref{1QGchar} and Proposition \ref{Tprop} that show the equivalence between $\|G_m(x)\|\le \|x\|$ and $\|G_{\phi(m)}(x)\|\le \|x\|$, we investigate the following condition for almost greedy-like properties

\begin{equation}\label{e1}
\|x-G_{\phi(m)}(x)\|\ \le\ \inf_{|A|\le m}\|x-P_A(x)\|,
\end{equation}
for all $x\in X$, $m\in\mathbb N$, and
$G_{\phi(m)}(x)\in\mathcal G(x,\phi(m))$. We then ask how
enlargement affects the characterization in Theorem~\ref{AGchar}. 

\begin{defi}\normalfont
A basis satisfying \eqref{e1} is called \emph{$\phi$-almost greedy with constant $1$}, abbreviated \phiAGone.
\end{defi}

A particular case of $\phi$-AG property appeared in \cite{C}, which considered the function $g(m):=\lceil\lambda m\rceil$, with a fixed $\lambda \ge 1$. The map $g$ belongs to $\Phi$, since it is strictly increasing and
$g(0)=0$.
\begin{prop}\cite[Corollary 3.2]{C}\label{p22}
Let $(e_n)_{n=1}^\infty$ be suppression quasi-greedy with constant $1$. For a fixed $\lambda \ge 1$, define $\phi(m) = \lceil\lambda m\rceil$, $m\in \mathbb{Z}_{\ge 0}$. Then $\mathcal{B}$ is \phiAGone{} if and only if it is $\phi$-symmetric for largest coefficients with constant $1$. 
\end{prop}

\textit{Symmetry for the largest coefficients} is a natural generalization of Property (A). 

\begin{defi}\label{defiSLC}\normalfont
A basis is \textit{$\phi$-symmetric for largest coefficients ($\phi$-SLC) with constant $C$} if
$$\|x+1_{\varepsilon A}\|\ \le\ C\|x+1_{\delta B}\|,$$
for all $x\in X$ with $\|x\|_\infty\le 1$, for all $A, B\in \mathbb{N}^{<\infty}$ with $\phi(|A|)\le |B|$ and pairwise disjoint $A, B, \supp(x)$, and for all signs $\varepsilon, \delta$. 
\end{defi}
We write \phiSLCone{} when $C=1$. Our first result addresses the two following questions: 
\begin{itemize}
\item what happens if we drop the suppression quasi-greedy assumption in Proposition~\ref{p22}? 
\item for which functions $\phi\in\Phi$ are \phiAGone{} and \phiSLCone{} equivalent for all bases?
\end{itemize}

\begin{thm}\label{M1}
Let $\phi\in\Phi$. The following are equivalent:
\begin{enumerate}
\item[(a)] For every semi-normalized Markushevich basis,
\phiAGone{} and \phiSLCone{} are equivalent.
\item[(b)] $\phi(1)\in\{1,2\}$.
\end{enumerate}
Moreover, \phiSLCone{} always implies \phiAGone, and if
$\phi(1)\ge3$, this implication is strict, even among Schauder bases
equivalent to the canonical basis of $\ell_1$.
\end{thm}

\begin{cor}\label{c24}
    For a fixed $\lambda \ge 1$, define $\phi(m) = \lceil\lambda m\rceil$, $m\in \mathbb{Z}_{\ge 0}$. The following are equivalent:
    \begin{enumerate}
    \item For every semi-normalized Markushevich basis, \phiAGone{} and \phiSLCone{} are equivalent. 
    \item The factor $\lambda$ is at most $2$.
    \end{enumerate}
\end{cor}

Theorem 1.12 sharpens and extends the earlier characterization in \cite[Corollary 3.2]{C}. For fixed-factor enlargements $\phi(m)=\lceil\lambda m\rceil$, that result established the equivalence between $\phi$-AG$_1$ and $\phi$-SLC$_1$ under the additional assumption of suppression quasi-greediness with constant $1$. Our theorem determines exactly when this assumption can be removed: the equivalence holds for all bases when $1\le\lambda\le2$, but fails in general when $\lambda>2$, even for Schauder bases equivalent to the canonical basis of $\ell_1$. More generally, we establish this sharp dichotomy for every $\phi\in\Phi$, showing that the equivalence depends only on $\phi(1)$, rather than on the subsequent growth of $\phi$.

A related threshold at $\lambda=2$ appears in [11, Theorem 2.8], but in a different setting. That theorem concerns an interval-based version of almost greediness, in which the greedy error is compared with coordinate projections onto intervals rather than arbitrary finite sets. Thus, although the two results share the same critical threshold, they characterize distinct approximation inequalities. Our theorem treats the original coordinate-projection comparison and extends the classification beyond fixed-factor enlargements to arbitrary functions $\phi\in\Phi$.

The converse in Theorem~\ref{M1} fails once $\phi(1)\ge3$.
To isolate an isometric condition that is implied by \phiAGone{}, we
introduce \weakphiSLCone{}, which is motivated by the observation that a greedy set of cardinality $\phi(m)$ and a competing set of cardinality at most $m$ may have $k$ indices in common. Removing these common indices leaves sets of cardinalities $\phi(m)-k$ and at most $m-k$, respectively, giving precisely the cardinality conditions in the following definition.

\begin{defi}\label{defiweakSLC}\normalfont
A basis is \textit{weak $\phi$-SLC with constant $C$} if
$$\|x+1_{\varepsilon A}\|\ \le\ C\|x+1_{\delta B}\|,$$
for all $x\in X$ with $\|x\|_\infty\le 1$, for all signs $\varepsilon, \delta$, and for all $A, B\in \mathbb{N}^{<\infty}$ such that
\begin{enumerate}
\item $A, B, \supp(x)$ are pairwise disjoint, and
\item there exist $m\ge 1$ and $k\ge 0$ so that $|A| \le m - k$ and $|B| = \phi(m)-k$. 
\end{enumerate}
\end{defi}
We write \weakphiSLCone{} when $C=1$.

\begin{thm}\label{M2}
Let $\phi\in\Phi$. The following are equivalent:
\begin{enumerate}
\item[(a)] For every semi-normalized Markushevich basis,
\phiAGone{} and \weakphiSLCone{} are equivalent.
\item[(b)] $\phi(1)\in\{1,2\}$.
\end{enumerate}
Moreover,  \phiAGone{} always implies \weakphiSLCone{}, and if
$\phi(1)\ge3$, this implication is strict, even among Schauder bases
equivalent to the canonical basis of $\ell_1$.
\end{thm}

Together, our first two theorems give the hierarchy
\[
 \phiSLCone\ \Longrightarrow\ \phiAGone\
 \Longrightarrow\ \weakphiSLCone.
\]
Both implications are strict exactly when $\phi(1)\ge3$.

Our final result characterizes \phiAGone{} bases, thus generalizing Theorem \ref{AGchar} to general functions $\phi$. To do so, we introduce right-hand
coefficient robustness in Section~\ref{section_weighted_weak} by strengthening the \weakphiSLCone{} to obtain the \weightedweakphiSLCone{} or equivalently, a homogeneous norm
inequality, called Property $(F_\phi)$. 

\begin{thm}\label{MainWeightedCharacterization}
Let $\phi\in\Phi$ and let $\mathcal B$ be a semi-normalized
Markushevich basis in a real or complex Banach space. The following
conditions are equivalent:
\begin{enumerate}
\item[(a)] $\mathcal B$ satisfies \phiAGone;
\item[(b)] $\mathcal B$ satisfies \weightedweakphiSLCone;
\item[(c)] $\mathcal B$ satisfies \Fphione.
\end{enumerate}
\end{thm}

Sections \ref{suff_con}, \ref{secWeakPhiSLC}, and \ref{section_weighted_weak} prove Theorems \ref{M1}, \ref{M2}, and \ref{MainWeightedCharacterization}, respectively.

\section*{Declaration of generative AI use}
The authors conceived the project, structured the manuscript, and independently formulated and proved Theorems \ref{M1} and \ref{M2} without AI assistance. After completing the initial draft, the authors used OpenAI's ChatGPT \cite{AI} to generate Theorem \ref{MainWeightedCharacterization} and related results in Section \ref{section_weighted_weak}, as well as to review mathematical arguments and identify typographical errors. The authors independently verified all AI-assisted mathematical arguments, corrected and revised the generated material as necessary, and take full responsibility for the accuracy and content of the paper.

\section{Symmetry for largest coefficients and $\phi$-almost greediness}\label{suff_con}
It turns out that \phiSLCone{} is relatively strong compared to \phiAGone{} mainly because of the following observation.
\begin{lem}\label{implySQG1}
A \phiSLCone{} basis is suppression quasi-greedy with constant $1$. 
\end{lem}
\begin{proof}
Pick $x\in X$ with $\|x\|_\infty\le 1$ and  $j\in \mathbb{N}\backslash \supp(x)$. In the definition of $\phi$-SLC, we can choose $A = \emptyset$ and $B = \{j\}$ to obtain 
$$\|x\|\ \le\ \|x+e_j\|.$$
Scaling and convexity give the same estimate when
$e_j$ is replaced by $t e_j$ for $|t|\ge\|x\|_\infty$.
Adding greedy coefficients in nondecreasing order of modulus now
gives suppression quasi-greediness with constant $1$.
\end{proof}

We shall use Lemma~\ref{implySQG1} to show that there are bases satisfying \eqref{e1} but failing \phiSLCone{}.
Before proving Theorem \ref{M1}, we need an easy reformulation of $\phi$-SLC.

\begin{prop}\label{ezcharSLC}
A basis is $\phi$-SLC with constant $C$ if and only if 
\begin{equation}\label{e12}\|x\|\ \le\ C\|x-P_A(x)+1_{\varepsilon B}\|,\end{equation}
for all $x\in X$ with $\|x\|_\infty\le 1$, for all $A, B\in \mathbb{N}^{<\infty}$ with $\phi(|A|)\le |B|$ and $B$ disjoint from $A\cup \supp(x)$, and for all signs $\varepsilon$. 
\end{prop}

\begin{proof}
Assume that our basis is $\phi$-SLC with constant $C$. Pick $x, A, B$, and $\varepsilon$ as in \eqref{e12}. By convexity, we have
$$\|x\|\ \le\ \sup_\delta\|x-P_A(x)+1_{\delta A}\|\ \le\ C\|x-P_A(x)+1_{\varepsilon B}\|.$$
Conversely, suppose that our basis satisfies \eqref{e12}. Pick $x, A, B, \varepsilon$, and $\delta$ as in Definition \ref{defiSLC} and set $y = x+1_{\varepsilon A}$. It follows that
$$\|x+1_{\varepsilon A}\|\ = \ \|y\|\ \le\ C\|y-P_A(y) + 1_{\delta B}\|\ =\ C\|x+1_{\delta B}\|.$$
This completes our proof. 
\end{proof}

\begin{lem}\label{l1}
If $\phi(1)\in\{1,2\}$, and $\mathcal{B}$ is \phiAGone{}, then $\mathcal{B}$ is suppression quasi-greedy with constant $1$. 
\end{lem}
\begin{proof}Pick a finitely supported $x\in X$ with $\|x\|_\infty\le 1$ and $j\notin \supp(x)$. Choose $A\in \mathbb{N}^{<\infty}$ such that $A\cap (\supp(x)\cup \{j\}) = \emptyset$ and $|A| = \phi(1)-1$. 
By \eqref{e1},
$$\|x\|\ =\ \|(x+e_j+1_A)-(e_j+1_A)\|\ \le\ \|(x+e_j+1_A)-1_A\|\ =\ \|x+e_j\|.$$
The inequality extends from finitely supported $x$ to all $x$ by
Lemma \ref{FiniteSupportReduction}. Iterating the resulting
coefficient-adding inequality along a greedy set shows that
$\mathcal{B}$ is suppression quasi-greedy with constant $1$.
\end{proof}

For $b\in \mathbb{F}$ and $b\neq 0$,  let $\sgn(b) = b/|b|$. 
We define the truncation operator $T_a:X\to X$ by the finite-set formula
\[
T_a(x)\ =\ x-P_{\Lambda_a(x)}(x)+a1_{\varepsilon\Lambda_a(x)},
\]
where $\Lambda_a(x)=\{n:|e_n^*(x)|>a\}$ and
$\varepsilon=(\sgn(e_n^*(x)))_n$. 

\begin{lem}\cite[Lemma 2.5]{BBG}\label{BBGlem} If a basis is suppression quasi-greedy with constant $C$, 
then $\|T_a(x)\|\le C\|x\|$ for every $x\in X$ and $a > 0$.
\end{lem}

We are now ready to prove Theorem \ref{M1}.

\begin{proof}[Proof of Theorem~\ref{M1}: (b) implies (a)]
Assume $\phi(1)\le2$. We shall show that a basis $\mathcal{B}$ is \phiAGone{} if and only if it is \phiSLCone{}. By Lemmas \ref{implySQG1} and \ref{l1}, $\mathcal{B}$ is suppression quasi-greedy with constant $1$. 

First, assume that $\mathcal{B}$ is \phiAGone{}. Pick $x, A, B, \varepsilon$, and $\delta$ as in Definition \ref{defiSLC}. If $A=\emptyset$, then $B$ is a greedy set for $x+1_{\delta B}$, so suppression quasi-greediness with constant $1$ gives $\|x\|\le\|x+1_{\delta B}\|$. Thus assume $A\ne\emptyset$. Set $y = x+ 1_{\varepsilon A} + 1_{\delta B}$, and choose $B'\subset B$ with $|B'| = \phi(|A|)$. We have
\begin{align*}
 \|x+1_{\varepsilon A}\|\ =\ \|y - P_B(y)\|  \ \le\ \|y-P_{B'}(y)\|\ \le\ \|y-P_A(y)\|\ =\ \|x+1_{\delta B}\|.
\end{align*}

Second, assume that $\mathcal{B}$ is \phiSLCone{}. Let $x\in X$, $m\in \mathbb{N}$, $A\in\mathbb{N}^{<\infty}$ with $|A| \le m$, and $B\in \mathcal{G}(x, \phi(m))$. Let $\alpha = \min_{n\in B}|e_n^*(x)|$. When $\alpha=0$,
totality of $(e_n^*)$ gives $x=P_B(x)$, so the desired estimate is
immediate. Otherwise, Condition \eqref{e20} gives
$\phi(|A\backslash B|)\le |B\backslash A|$. Let $\varepsilon=(\sgn(e_i^*(x)))_{i\in B}$. By Proposition \ref{ezcharSLC} and Lemma \ref{BBGlem}, we have
\begin{align*}
\|x-P_B(x)\| &\ \le\ \|x-P_{B}(x) - P_{A\backslash B}(x) + \alpha 1_{\varepsilon B\backslash A}\|\\
&\ =\ \|T_\alpha(x-P_{A}(x))\|\ \le\ \|x-P_A(x)\|. 
\end{align*}
This completes our proof. 
\end{proof}

To complete the proof of Theorem~\ref{M1}, suppose $\phi(1)\ge3$. We construct a basis that
\begin{enumerate}
\item [a)] is \phiAGone{} but 
\item [b)] is not $1$-suppression quasi-greedy and thus, is not \phiSLCone{}. 
\end{enumerate}

Let $X$ be the completion of $c_{00}$ with respect to the following norm
$$\|x = (x_i)_{i=1}^\infty\|\ =\ \left|x_1 + \frac{x_2}{3}\right| \vee \left|\frac{x_1}{3} + x_2\right| \vee \frac{1}{3}\sum_{i=1}^\infty|x_i|.$$
Clearly, $X$ is a renorming of $\ell_1$, and the canonical basis $\mathcal{B} = (e_n)_{n=1}^\infty$ is semi-normalized. 

\begin{claim}
The basis $\mathcal{B}$ is not $1$-suppression quasi-greedy. 
\end{claim}

\begin{proof}
Consider the vectors $x = (1, -5/6, 0, 0, \ldots)$ and $y = (0, -5/6, 0, 0, \ldots)$. We have $\|y\| = 5/6$, while
$$\|x\|\ =\ \max\left\{\frac{13}{18},\frac{1}{2},\frac{11}{18}\right\}\ =\ \frac{13}{18}.$$
Since $\|y\| > \|x\|$, $\mathcal{B}$ is not $1$-suppression quasi-greedy.   
\end{proof}

\begin{claim}
The basis $\mathcal{B}$ is \phiAGone{}.     
\end{claim}

\begin{proof}
Let $x = (x_i)_{i=1}^\infty\in X$, $m\in \mathbb{N}$, $A\in \mathcal{G}(x, \phi(m))$, and $B\in \mathbb{N}^{<\infty}$ with $|B|\le m$. We shall show that 
$$\|x-P_A(x)\|\ \le\ \|x-P_B(x)\|.$$
We proceed by case analysis.

Case 1: $\{1, 2\}\subset A$.  Then
$$\|x-P_A(x)\|\ =\ \frac{1}{3}\sum_{\substack{i\ge 1\\ i\notin A}}|x_i|\ \le\  \frac{1}{3}\sum_{\substack{i\ge 1\\ i\notin B}}|x_i|\ \le\ \|x-P_B(x)\|.$$

Case 2: $|\{1, 2\}\cap A| = 1$. Without the loss of generality, assume that $1\in A$ and $2\notin A$. In this case,
$$\|x-P_A(x)\|\ =\ |x_2| \vee \frac{1}{3}\sum_{\substack{i\ge 1\\ i\notin A}}|x_i|.$$
As in Case 1, $3^{-1}\sum_{\substack{i\ge 1\\ i\notin A}}|x_i|\le \|x-P_B(x)\|$. We need only to show that
$$\|x-P_B(x)\|\ \ge\ |x_2|.$$
Since $A\in \mathcal{G}(x, \phi(m))$ and $2\notin A$, we know that there are at least $\phi(m)+1$ coefficients $x_i$ whose modulus is at least $|x_2|$. Hence,
$$\|x-P_B(x)\|\ \ge\ \frac{\phi(m)+1-m}{3}|x_2|\ \ge\ |x_2|,$$
where the last inequality follows from $\phi(m) - m\ge\phi(1)-1\ge 2$ for all $m\in \mathbb{N}$. 

Case 3: $\{1, 2\}\cap A = \emptyset$. Then
$$\|x-P_A(x)\|\ =\ \left|x_1 + \frac{x_2}{3}\right|\vee \left|\frac{x_1}{3} + x_2\right|\vee \frac{1}{3}\sum_{\substack{i\ge 1\\ i\notin A}}|x_i|.$$
As before, we need only to verify that
\begin{equation}\label{e14}\|x-P_B(x)\|\ \ge\  \left|x_1 + \frac{x_2}{3}\right|\vee \left|\frac{x_1}{3} + x_2\right|\ =:\ g(x_1, x_2).\end{equation}
Let $a = \max\{|x_1|, |x_2|\}$. 
Since $A\in \mathcal{G}(x, \phi(m))$ and $\{1, 2\}\cap A = \emptyset$, there are at least $\phi(m)+1$ coefficients $x_i$ with $|x_i|\ge a$. If $x_1x_2\le 0$, then $g(x_1, x_2)\le a$. Then 
$$\|x-P_B(x)\|\ \ge\ \frac{\phi(m)+1-m}{3}a\ \ge\ a\ \ge\ g(x_1, x_2).$$
For the rest of the proof, by replacing $x$ with $-x$ if necessary, we may assume $x_1,x_2>0$. 

\begin{itemize}
\item Case 3a: $B\cap \{1, 2\} = \emptyset$. It follows immediately that
$\|x-P_B(x)\|\ \ge\ g(x_1, x_2)$.
\item Case 3b: $\{1, 2\}\subset B$. Note that since $\{1, 2\}\cap A = \emptyset$, there exist at least $\phi(m)$ coefficients $x_i$ with $|x_i|\ge a$ and $i\ge 3$. It follows that
$$\|x-P_B(x)\|\ \ge\ \frac{\phi(m)-(m-2)}{3}a\ \ge\ \frac{4a}{3}\ \ge\ \left(x_1 + \frac{x_2}{3}\right)\vee \left(\frac{x_1}{3} + x_2\right).$$
\item Case 3c: $|\{1, 2\}\cap B| = 1$. We may assume that $1\in B$ and $2\notin B$. We have
\begin{align*}\|x-P_B(x)\|\ \ge\ \frac{x_2}{3} + \frac{1}{3}\sum_{\substack{i\ge 3\\ i\notin B}}|x_i|&\ \ge\ \frac{x_2}{3} + \frac{\phi(m)-(m-1)}{3}a\\
&\ \ge\ \frac{x_2}{3} + a,\end{align*}
which can be easily seen to be at least $g(x_1, x_2)$. 
\end{itemize}
This completes our proof that $\mathcal{B}$ is \phiAGone{}. 
\end{proof}

\begin{proof}[Completion of the proof of Theorem~\ref{M1}]
The preceding counterexample applies whenever $\phi(1)\ge3$.
We have a renorming of $\ell_1$, where the canonical basis is $\phiAGone$ but not $\phiSLCone$. Together with the first part of the proof,
this establishes Theorem \ref{M1}.
\end{proof}

\section{Weak symmetry for largest coefficients}\label{secWeakPhiSLC}
By Theorem~\ref{M1}, \phiSLCone{} implies \phiAGone{}
for every $\phi\in\Phi$, while the converse may fail. This motivates
the weaker notion introduced below. Our aim is to establish the hierarchy
\[
 \phiSLCone\ \Longrightarrow\ \phiAGone
 \ \Longrightarrow\ \mbox{weak } \phiSLCone.
\]

We have an equivalent definition of weak $\phi$-SLC. 

\begin{defi}\label{defiweakSLC_2}\normalfont
A basis is \textit{weak $\phi$-SLC with constant $C$} if
$$\|x+1_{\varepsilon A}\|\ \le\ C\|x+1_{\delta B}\|,$$
for all $x\in X$ with $\|x\|_\infty\le 1$, for all signs $\varepsilon, \delta$, and for all $A, B\in \mathbb{N}^{<\infty}$ such that
\begin{enumerate}
\item $A, B, \supp(x)$ are pairwise disjoint, and
\item there exists $m\ge 1$ with $|B|-|A|\ge \phi(m)-m$ and $|B|\le \phi(m)$.
\end{enumerate}
\end{defi}

\begin{lem}\label{lemmaequiv1}
    Definition \ref{defiweakSLC_2} is equivalent to Definition \ref{defiweakSLC}.
\end{lem}

\begin{proof}
    If there exists $m\ge 1$ and $k\ge 0$ so that $|A|\le m-k$ and $|B| = \phi(m)-k$, then
    $$|B| - |A|\ \ge\ (\phi(m)-k)-(m-k)\ =\ \phi(m) - m\mbox{ and }|B|\ \le\ \phi(m).$$

    Conversely, if there exists $m\ge 1$ so that $|B| - |A|\ge \phi(m)-m$ and $|B|\le \phi(m)$, 
    then $\phi(m) - |B|\le m - |A|$. Let $k:= \phi(m)-|B|\ge 0$. Then $k\le m-|A|$, so $|A|\le m-k$. 
\end{proof}

We record several facts about weak $\phi$-SLC. 
\begin{prop}\label{weakSLC_properties} The following statements
hold for every $\phi\in \Phi$.
\begin{enumerate}
\item [i)] If a basis is $\phi$-AG with constant $C$, then it is weak $\phi$-SLC with constant $C$.  
\item [ii)] If $\phi(1)\ge3$, \weakphiSLCone{} need not imply suppression quasi-greediness with constant $1$. This differs from \phiSLCone{}.
\item [iii)] If $\phi(1)\le 2$, then a \weakphiSLCone{} basis is suppression quasi-greedy with constant $1$.
\item [iv)] A basis is weak $\phi$-SLC with constant $C$ if and only if 
$$\|x\|\ \le\ C\|x-P_A(x) + 1_{\varepsilon B}\|,$$
for all $x\in X$ with $\|x\|_\infty\le 1$, for all signs $\varepsilon$, and for all $A, B\in \mathbb{N}^{<\infty}$ such that
\begin{enumerate}
\item $B$ is disjoint from $A\cup \supp(x)$, and
\item there exist $m\ge 1$ and $k\ge 0$ so that $|A| \le m - k$ and $|B| = \phi(m)-k$. 
\end{enumerate}
\item [v)] If a basis is \weakphiSLCone{} and satisfies $\|T_a(x)\|\le \|x\|$ for all truncation operators $T_a$, then the basis is \phiAGone{}. 
\item [vi)] If $\phi(1)\le 2$, then a basis is \weakphiSLCone{} if and only if it is \phiAGone{}. 
\end{enumerate}
\end{prop}

\begin{proof}
We prove i). Suppose that our basis is $\phi$-AG with constant $C$.
By Lemma \ref{FiniteSupportReduction}, it suffices to verify the
inequality for finitely supported $x$: for a general $x$, we first
approximate it by finitely supported vectors avoiding $A\cup B$ and
rescale them by factors converging to $1$ to ensure the coefficient
bound $\|x\|_\infty\le 1$.
Choose $A, B, x, \varepsilon$, and $\delta$ as in Definition \ref{defiweakSLC}, with $x$ finitely supported. Let $y = x + 1_{\varepsilon A} + 1_{\delta B} + 1_D$, where $D\in \mathbb{N}^{<\infty}$ is disjoint from $\supp(x)\cup A\cup B$ and $|D| = k$. Since $|B\cup D| = \phi(m)$ and $|A\cup D|\le m$, the set
$B\cup D$ is greedy for $y$, and we have
$$\|x+1_{\varepsilon A}\|\ =\ \|y - P_{B\cup D}(y)\|\ \le\ C\|y-P_{A\cup D}(y)\|\ =\ \|x+1_{\delta B}\|.$$

Statement ii) follows from i) and the counterexample in the proof of Theorem~\ref{M1}.

We prove iii). Pick $x\in X$ with $\|x\|_\infty\le 1$ and $j\in\mathbb{N}\backslash \supp(x)$. If $\phi(1) = 2$, then we choose $(m,k) = (1,1)$ in Definition \ref{defiweakSLC}. If 
$\phi(1) = 1$, then we choose $(m,k) = (1,0)$ in Definition \ref{defiweakSLC}. In both cases, we have
$$\|x\|\ \le\ \|x+e_j\|.$$
By the same reasoning as in the proof of Lemma \ref{implySQG1}, the basis is suppression quasi-greedy with constant $1$.

The proof of iv) is similar to the proof of Proposition \ref{ezcharSLC}.

Next, we prove v). Let $x\in X$, $m\in \mathbb{N}$, $A\in \mathcal{G}(x, \phi(m))$, and $B\in \mathbb{N}^{<\infty}$ with $|B| \le m$. 
Let $\alpha = \min_{n\in A}|e_n^*(x)|$. If $\alpha=0$, then $x=P_A(x)$ by totality, and the conclusion is immediate. Otherwise, set $\varepsilon=(\sgn(e_i^*(x)))_{i\in A}$. Observe that 
$$\begin{cases}|B\backslash A| \ =\ |B| - |A\cap B|\ \le\ m - |A\cap B|,\\ |A\backslash B| \ =\ |A| - |A\cap B|\ =\ \phi(m) - |A\cap B|.\end{cases}$$
It follows from iv) and by $\|T_\alpha\|\le 1$ that
\begin{align*}
\|x-P_A(x)\|&\ \le\ \|x-P_A(x) - P_{B\backslash A}(x) + \alpha 1_{\varepsilon A\backslash B}\|\\
&\ =\ \|T_\alpha(x-P_B(x))\|\ \le\ \|x-P_B(x)\|.
\end{align*}
This concludes the proof of v).

Lastly, vi) follows from i), iii), v), and Lemma \ref{BBGlem}.
\end{proof}

\begin{cor}
Let $(e_n)_{n=1}^\infty$ be a basis. 
    If $\phi(1)\le 2$, then the following are equivalent:
    \begin{enumerate}
        \item[(a)] The basis $(e_n)_{n=1}^\infty$ is \phiSLCone{}.
        \item[(b)] The basis $(e_n)_{n=1}^\infty$ is \phiAGone{}. 
        \item[(c)] The basis $(e_n)_{n=1}^\infty$ is \weakphiSLCone{}. 
    \end{enumerate}
\end{cor}

\begin{proof}
    Due to Theorem \ref{M1}, (a) is equivalent to (b). By Proposition \ref{weakSLC_properties}, item (vi), we have (b) is equivalent to (c). 
\end{proof}

\begin{thm}\label{weakNotAGExample}
    If $\phi\in \Phi$ and $\phi(1)\ge 3$, then there exists a normalized basis that is equivalent to the canonical basis of $\ell_1$, \weakphiSLCone{}, but not \phiAGone{}. 
\end{thm}

\begin{proof}
Let $N:=\phi(1)$ and $I:=\{1,\dots, N\}$, and define 
$$
\mathcal{A}\ :=\ \left\lbrace  a=(a_n)_{n\in I}\subset \mathbb{R}\,:\, |a_n|\le 1, \forall n\in I, |\supp(a)|\le 3,  \left\vert \frac{a_1}{2}+\sum_{n\in I}a_n \right\vert \le \frac{1}{2}\right\rbrace.
$$
Notice that $\mathcal{A}=-\mathcal{A}$ and 
\begin{align}
&(a_1,\dots, a_N)  \in \mathcal{A}\ \Longleftrightarrow\ (a_1, a_{\pi(2)}, \dots, a_{\pi(N)})\in \mathcal{A},\nonumber\\
&\forall \mbox{ bijections } \pi: \{2,\dots, N\}\rightarrow \{2,\dots, N\}. \label{symmetric} 
\end{align}
Furthermore,
\begin{align}
&(-1,a_2,\dots, a_N) \in\mathcal{A}\nonumber\\
&\forall (a_2,\dots, a_N)\mbox{ with }a_j\in \{-1,0,1\}, |\{j: a_j\not=0\}|\le 2, \sum_{2\le n\le N}a_n \ge 1. \label{A1} 
\end{align}
(We can substitute $-1$ for $1$ and  $\sum_{2\le n\le N}a_n \le -1$ for $ \sum_{2\le n\le N}a_n \ge 1$ above.) \\
Define on $\mathtt{c}_{00}$ over $\mathbb{R}$ the following seminorms: for $x=(x_n)_{n\in \mathbb{N}}$, 
\begin{align*}
\|x\|_{\triangle}&\ :=\ \sum_{n\not\in I}|x_n|,\\
\|x\|_{\square}&\ :=\ \frac{1}{3}\max_{n\in I}|x_n|,\mbox{ and}\\
\|x\|_{\circ}&\ :=\ \sup_{a\in \mathcal{A}}\left\vert \sum_{n\in I} a_nx_n\right\vert.
\end{align*}
Let $X$ be the completion of $\mathtt{c}_{00}$ with the norm 
$$
\|x\|\ :=\ \|x\|_{\triangle}+\max\{\|x\|_{\square}, \|x\|_{\circ}\}.
$$
Let $(e_n)_{n=1}^{\infty}$ be the canonical vector system of $X$, and $(e_n^*)_{n=1}^{\infty}$ the canonical biorthogonal functionals. Note that $(e_n)_{n=1}^{\infty}$ is normalized, equivalent to the canonical unit vector basis of $\ell_1$, and by \eqref{A1}, 
\begin{equation}
\|1_{\delta D}\|_{\circ}\ \ge\ 1, \mbox{ for all nonempty sets } D\subset I\mbox{ and signs }\delta. \label{A2}
\end{equation}
Indeed, 
\begin{itemize}
    \item if $1\notin D$, then suppose that $j\in D$. To witness \eqref{A2}, we use 
    $$a \ =\ (-1, 0, \ldots, 0, 1, 0, \ldots, 0)\in\mathbb{R}^N\mbox{ with }1\mbox{ in the }j\mbox{th position};$$
    \item if $1\in D$, and $\{2, \ldots, N\}\not\subset D$, then suppose that $j\notin D$. To witness \eqref{A2}, we again use 
    $$a \ =\ (-1, 0, \ldots, 0, 1, 0, \ldots, 0)\in\mathbb{R}^N\mbox{ with }1\mbox{ in the }j\mbox{th position};$$
    \item if $D = \{1, 2, \ldots, N\}$, then without loss of generality, assume $(1_{\delta D})_1 = 1$.  We use 
    $$a\ =\ (-1, 1, 1, 0,\ldots, 0).$$
\end{itemize}

Recall that $\phi(m)-m$ is a nondecreasing function. Hence, due to Lemma~\ref{lemmaequiv1}, to prove that $(e_n)_{n=1}^{\infty}$ is \weakphiSLCone{}, it suffices to show that 
\begin{equation}\label{e2}
\|x + 1_{\varepsilon A}\|\ \le\ \|x + 1_{\delta B}\|, 
\end{equation}
for all $m\in \mathbb{N}$, all $\varepsilon, \delta$, and all $A, B$ finite subsets of $\mathbb{N}$ with 
\begin{itemize}
    \item $\|x\|_\infty\le 1$,
    \item $\supp(x)$, $A$, and $B$ are pairwise disjoint, and 
    \item $|B| - |A|\ge N-1$. 
\end{itemize}
To show this, fix $x, A, B, \varepsilon, \delta$ as above, and let 
\begin{align*}
&B_0\ :=\ B\cap I,\quad B_1\ :=\ B\backslash I;\\
&A_0\ :=\ A\cap I,\quad A_1\ :=\ A\backslash I.
\end{align*}
It follows at once from definitions that 
\begin{equation}
\|x+1_{\varepsilon A}\|_{\triangle}\ =\ \|x\|_{\triangle}+|A_1|\quad\text{ and }\quad \|x+1_{\delta B}\|_{\triangle}\ =\ \|x\|_{\triangle}+|B_1| \label{estimatesell1}
\end{equation}
and 
\begin{equation}
\|x+1_{\varepsilon A}\|_{\circ}\ \le\  3. \label{upperboundforfirstN}
\end{equation}
Now we proceed by case analysis; we will omit the estimates for $\|x+1_{\varepsilon A}\|_{\square}$, as they are clear. 

\bigskip

Case 1: If $|B_1|\ge |A_1|+3$, then \eqref{e2} follows from \eqref{estimatesell1} and \eqref{upperboundforfirstN}. 

\bigskip

Case 2: If $|B_1|=|A_1|+2$, then $|B_0|\ge |A_0|+N-3$. Then since $|A_0| + |B_0| \le N$, we have
$$|A_0|\ \le\ 1,$$ 
and since $|B_0| +|I\cap (\supp(x+1_{\varepsilon  A}))| \le N$, 
$$\quad |I\cap (\supp(x+1_{\varepsilon  A}))|\ \le\ 3.$$ We consider two subcases.
\begin{itemize}
\item If $\|x+1_{\varepsilon A}\|_{\circ}\le 2$, using \eqref{estimatesell1}, we obtain \eqref{e2}.
\item If $2<\|x+1_{\varepsilon A}\|_{\circ}\le 3$,  then $|I\cap (\supp(x+1_{\varepsilon  A}))|= 3$. Given that 
$$|B_0|\ \ge\ N-3+|A_0|\ \ge\ N-3,$$ 
we know that $|B_0|=N-3$ and $A_0=\emptyset$. Hence, 
$$\|x\|_{\circ}\ =\ \|x+1_{\varepsilon A}\|_{\circ}\ >\ 2.$$ 
Choose $a\in \mathcal{A}$ so that 
$$
2\ <\ \|x\|_{\circ}\ =\ \sum_{n\in  \supp(a)\cap \supp(x) }a_ne_n^*(x).  
$$
As $|\supp(x)\cap I|=3$ and $|\supp(a)|\le 3$, it follows that $\supp(x)\cap I=\supp(a)$. Hence, 
\begin{align*}
\|x+1_{\delta B} \|_{\circ}&\ \ge\ \sum_{n\in  \supp(a) }a_ne_n^*(x+1_{\delta B})\\
&\ =\  \sum_{n\in \supp(a) }a_ne_n^*(x)\ =\ \|x\|_{\circ}\ =\ \|x+1_{\varepsilon A}\|_{\circ}.
\end{align*}
\end{itemize}

\bigskip

Case 3: If $|B_1|=|A_1|+1$, then $|B_0|\ge |A_0|+N-2$. 
Hence, 
$$|I\cap\supp(x)| + |A_0| \ \le\ N - |B_0|\ \le\ N - (|A_0| + N-2)\ =\ 2 - |A_0|\ \le\ 2.$$
It follows that $\|x+1_{\varepsilon A}\|_{\circ}\le 2$. We consider two subcases. 
\begin{itemize}
\item If $\|x+1_{\varepsilon A}\|_{\circ}\le 1$, then using \eqref{estimatesell1}, we obtain \eqref{e2}. 
\item If $1<\|x+1_{\varepsilon A}\|_{\circ}\le 2$, then $|I\cap (\supp(x+1_{\varepsilon A}))|=2$, and $|B_0|=N-2$. Hence, $A_0=\emptyset$. Choose $a\in \mathcal{A}$ so that 
$$
1\ <\ \|x\|_{\circ}\ =\ \sum_{n\in  \supp(a)\cap \supp(x) }a_ne_n^*(x).  
$$
As $|\supp(a)|\le 3$, it follows that $|\supp(a)\cap B|\le 1$. Hence, 
$$
\|x+1_{\delta B} \|_{\circ}\ \ge\ \sum_{n\in  \supp(a)}a_ne_n^*(x+1_{\delta B})\ \ge\  \|x\|_{\circ}-1\ =\ \|x+1_{\varepsilon A}\|_{\circ}-1. 
$$
From this and \eqref{estimatesell1}, we derive \eqref{e2}. 
\end{itemize}

\bigskip

Case 4: If $|B_1|=|A_1|$, then $|B_0|\ge |A_0|+N-1$. Hence, $A_0=\emptyset$, and $|I\backslash B_0|\le 1$. We consider three subcases. 
\begin{itemize}
\item If $\supp(x)\cap I=\emptyset$, then $\|x+1_{\varepsilon A}\|_{\circ}=0$, and \eqref{e2} follows from \eqref{estimatesell1}. 
\item If $\supp(x)\cap I=\{1\}$, there is $0<|t|\le 1$ such that 
$$
P_I(x)\ =\ P_I(x+1_{\varepsilon A})\ =\ te_1\quad\text{ and }\quad P_I(x+1_{\delta B})\ =\ t e_1+\sum_{n=2}^{N}\delta_n e_n.
$$
Hence, $\|x+1_{\varepsilon A}\|_{\circ}=|t|$, whereas using \eqref{A1}, we get 
$$
\|x+1_{\delta B}\|_{\circ}\ \ge\ |-t+\delta_2+\delta_3|\ \ge\ |t|.  
$$
Thus, an application of \eqref{estimatesell1} gives \eqref{e2}. 
\item If there is $2\le k\le N$ such that $\supp(x)\cap I=\{k\}$, there is $0<|t|\le 1$ such that 
$$
P_I(x)\ =\ P_I(x+1_{\varepsilon A})\ =\ te_k\quad \text{ and }\quad P_I(x+1_{\delta B})\ =\ t e_k+\sum_{\substack{1\le n\le N\\n\not=k}}\delta_n e_n.
$$
Hence, $\|x+1_{\varepsilon A}\|_{\circ}=|t|$ and, choosing $j\in I\backslash \{1,k\}$ and using \eqref{A1}, we get 
$$
\|x+1_{\delta B}\|_{\circ}\ \ge\ |-\delta_1+t+\delta_j|\ \ge\ |t|.  
$$
Thus, an application of \eqref{estimatesell1} gives \eqref{e2}. 
\end{itemize}

\bigskip

Case 5: If $|B_1|=|A_1|-1$, then 
$$|B_0| \ =\  |B|-|B_1|\ \ge\ (|A| + N-1) - (|A_1|-1)\ =\ |A_0| +N\ \ge\ N.$$
Hence, $I = B_0\subset B$. By \eqref{A2}, 
$$
\|x+1_{\varepsilon A}\|_{\circ}\ =\ 0 \quad \text{ and }\quad \|x+1_{\delta B}\|_{\circ}\ =\ \|1_{\delta I}\|_{\circ}\ \ge\ 1, 
$$
which, when combined with \eqref{estimatesell1}, gives \eqref{e2}.

As we have considered all possible cases, we conclude that \eqref{e2} holds.

To see that \eqref{e1} does not hold, let 
$$x\ :=\ \frac{e_1}{2}+1_{I}+e_{N+1}, \quad B\ :=\ \{N+1\}\cup I\backslash \{2\}, \quad \mbox{ and }\quad A\ :=\ \{N+1\}.$$
Then $B\in \mathcal{G}(x,N)$ and $|A|= 1$, but 
$$
\|x-P_A(x)\|\ =\ \left\|\frac{e_1}{2}+1_{I}\right\|\ \le\ \frac{1}{2},
$$
whereas 
\begin{align*}
\|x-P_B(x)\|\ =\ \|e_2\|\ =\ 1.
\end{align*}
\end{proof}

\begin{proof}[Proof of Theorem~\ref{M2}]
The implication \phiAGone{} $\Longrightarrow$ \weakphiSLCone{} is
Proposition~\ref{weakSLC_properties}, item i). If $\phi(1)\le2$, the reverse
implication follows from Proposition~\ref{weakSLC_properties}, item vi).
If $\phi(1)\ge3$, Theorem~\ref{weakNotAGExample} supplies a
normalized Schauder basis that is equivalent to the canonical basis of $\ell_1$, $\weakphiSLCone$, but not
\phiAGone. 
\end{proof}

\section{Weighted weak symmetry and $\phi$-almost greediness}
\label{section_weighted_weak}

By Theorem~\ref{M2}, weak $\phi$-SLC need not imply
\phiAGone{}. We now strengthen weak symmetry by allowing
coefficients of arbitrary modulus at least $1$ on the right-hand
side, while retaining the same cardinality restrictions.

\begin{defi}\label{defWeightedWeakSLC}\normalfont
A basis $\mathcal B$ satisfies \emph{\weightedweakphiSLCone} if
\begin{equation}\label{weightedWeakSLC}
 \left\|x+1_{\varepsilon A}\right\|
 \ \le\ \left\|x+\sum_{n\in B}a_ne_n\right\|
\end{equation}
for every $x\in X$ with $\|x\|_\infty\le1$, finite sets
$A,B\subset\mathbb N$ such that $A,B,\supp(x)$ are pairwise
disjoint, $|a_n|\ge 1$ for each $n\in B$, and for some $m\ge1$,
\begin{equation}\label{weightedWeakSizes}
 |B|-|A|\ \ge\ \phi(m)-m,\qquad |B|\ \le\ \phi(m),
\end{equation}
and every choice of signs $\varepsilon$ on $A$ and scalars
$(a_n)_{n\in B}$ as above.
\end{defi}

Taking $|a_n|=1$ in \eqref{weightedWeakSLC} recovers
\weakphiSLCone{}. We shall prove that the weighted condition is
precisely what is needed to characterize \phiAGone{}.
There is an equivalent homogeneous formulation.

\begin{defi}\label{defFphi}\normalfont
A basis $\mathcal B$ satisfies \emph{Property \Fphione} if
\begin{equation}\label{Fphi_inequality}
 \|x\|\ \le\ \|x-P_A(x)+u\|
\end{equation}
whenever $x,u\in X$, $A\subset\mathbb N$ is finite,
$B=\supp(u)$ is nonempty, finite, and disjoint from $A\cup\supp(x)$,
$A,B$ satisfy \eqref{weightedWeakSizes} for some $m\ge1$, and
\begin{equation}\label{Fphi_coefficients}
 \min_{n\in B}|e_n^*(u)|\ \ge\ \|x\|_\infty.
\end{equation}
\end{defi}

\begin{rek}\normalfont
Property $(F_\phi)_1$ gives a homogeneous form of weighted weak
$\phi$-SLC and is analogous to the norm comparison in Proposition~\ref{ezcharSLC} and 
Proposition~\ref{weakSLC_properties}, item iv). It is related in spirit
to functional characterizations of almost greediness, including
Property $(F^*)$ in \cite{BM}. 
\end{rek}

\begin{prop}\label{FphiWeightedEquivalence}
For every $\phi\in\Phi$, a semi-normalized Markushevich basis
satisfies \weightedweakphiSLCone{} if and only if it satisfies
\Fphione.
\end{prop}
\begin{proof}
Suppose first that \Fphione{} holds. Let
$x,A,B,\varepsilon,(a_n)_{n\in B}$ be admissible for
\eqref{weightedWeakSLC} and put
\[
 z\ =\ x+1_{\varepsilon A},\qquad u\ =\ \sum_{n\in B}a_ne_n.
\]
Then $\|z\|_\infty\le1$,
$P_A(z)=1_{\varepsilon A}$,
and $\min_{n\in B}|a_n|\ge1$ when $B\ne\emptyset$.
Thus, applying \eqref{Fphi_inequality} to $z,u,A$ gives
\[
 \|x+1_{\varepsilon A}\|\ =\ \|z\|
 \ \le\ \|z-P_A(z)+u\|
 \ =\ \left\|x+\sum_{n\in B}a_ne_n\right\|.
\]

Conversely, assume that the weighted weak $\phi$-SLC inequality
holds. Fix $x,u,A,B$ as in Definition~\ref{defFphi}. Set
\[
 \alpha\ =\ \min_{n\in B}|e_n^*(u)|\ >\ 0,
 \qquad y\ =\ x-P_A(x).
\]
We have $\|y\|_\infty\le\|x\|_\infty\le\alpha$.
Since $A$ is finite and each coefficient of $P_A(x)$ has modulus
at most $\alpha$, convexity of the norm gives
\begin{equation}\label{Fphi_convexity}
 \|x\|\ =\ \|y+P_A(x)\|
 \ \le\ \sup_{\varepsilon}\|y+\alpha1_{\varepsilon A}\|.
\end{equation}
The pair $A,B$ has exactly the cardinality restrictions of
\eqref{weightedWeakSLC}. For every sign $\varepsilon$, apply
\eqref{weightedWeakSLC} to $\alpha^{-1}y$ and the coefficients
$a_n=\alpha^{-1}e_n^*(u)$, $n\in B$. Multiplying by $\alpha$ gives
\[
 \|y+\alpha1_{\varepsilon A}\|\ \le\ \|y+u\|.
\]
Together with \eqref{Fphi_convexity}, this yields
\eqref{Fphi_inequality}.
\end{proof}

\begin{proof}[Proof of Theorem~\ref{MainWeightedCharacterization}]
By Proposition~\ref{FphiWeightedEquivalence}, it remains to show
that \phiAGone{} is equivalent to \weightedweakphiSLCone{}.

Assume first \phiAGone{}. Fix $x,A,B,\varepsilon$ and
$(a_n)_{n\in B}$ as in Definition~\ref{defWeightedWeakSLC}.
We begin with $x$ finitely supported. By the proof of  Lemma~\ref{lemmaequiv1}, there are integers $m\ge1$ and $k\ge0$
such that
\[
 |A|\ \le\ m-k,\qquad |B|\ =\ \phi(m)-k.
\]
Choose $D\subset\mathbb N$ disjoint from
$\supp(x)\cup A\cup B$ with $|D|=k$, and set
\[
 v\ =\ x+1_{\varepsilon A}+\sum_{n\in B}a_ne_n+1_D.
\]
Since $\|x\|_\infty\le1$ and $|a_n|\ge1$, the set
$B\cup D$ is greedy for $v$ of order $\phi(m)$. Moreover,
$|A\cup D|\le m$. Hence
\begin{align*}
 \|x+1_{\varepsilon A}\|
 &\ =\ \|v-P_{B\cup D}(v)\|\\
 &\ \le\ \|v-P_{A\cup D}(v)\|\\
 &\ =\ \left\|x+\sum_{n\in B}a_ne_n\right\|,
\end{align*}
which is \eqref{weightedWeakSLC}.

For general $x$, Lemma~\ref{FiniteSupportReduction} gives
finitely supported $x_r\to x$ whose supports avoid $A\cup B$ and
such that $\|x_r-x\|_\infty\to0$. Set
\[
 \eta_r\ =\ \max\{1,\|x_r\|_\infty\},\qquad
 y_r\ =\ \eta_r^{-1}x_r.
\]
Then $\|y_r\|_\infty\le1$ and $y_r\to x$, because
$\eta_r\to1$. Applying the finite-support argument to $y_r$ and
passing to the limit proves \eqref{weightedWeakSLC} for $x$.

Conversely, assume \weightedweakphiSLCone{}.
Let $v\in X$, $m\ge1$, $G\in\mathcal G(v,\phi(m))$, and
$F\subset\mathbb N$ with $|F|\le m$. We must show
\[
 \|v-P_G(v)\|\ \le\ \|v-P_F(v)\|.
\]
Put $\alpha=\min_{n\in G}|e_n^*(v)|$.
If $\alpha=0$, then $v=P_G(v)$ by totality of the
coordinate functionals, so the inequality is immediate.
If $G\setminus F=\emptyset$, then $G=F$ and there is again
nothing to prove. Otherwise, set
\[
 k\ =\ |G\cap F|,\qquad A\ =\ F\setminus G,\qquad
 B\ =\ G\setminus F,\qquad z\ =\ v-P_{F\cup G}(v).
\]
We have $|A|\le m-k$, $|B|=\phi(m)-k$,
$A,B,\supp(z)$ are pairwise disjoint, and
\[
 \|z\|_\infty\ \le\ \alpha,
 \qquad |e_n^*(v)|\ \le\ \alpha\quad(n\in A).
\]
By convexity and \eqref{weightedWeakSLC}, applied to
$\alpha^{-1}z$ with coefficients
$a_n=\alpha^{-1}e_n^*(v)$ on $B$, we obtain
\begin{align*}
 \|v-P_G(v)\|
 &\ =\ \|z+P_A(v)\|\\
 &\ \le\ \sup_{\varepsilon}\|z+\alpha1_{\varepsilon A}\|\\
 &\ \le\ \|z+P_B(v)\|\\
 &\ =\ \|v-P_F(v)\|.
\end{align*}
This proves \phiAGone{} and completes the proof.
\end{proof}

\begin{cor}\label{strict_weighted_weak}
For every $\phi\in\Phi$, \weightedweakphiSLCone{} implies
\weakphiSLCone{}. These two properties are equivalent for all
semi-normalized Markushevich bases if and only if
$\phi(1)\in\{1,2\}$. If $\phi(1)\ge3$, the implication is
strict, even among Schauder bases equivalent to the canonical
basis of $\ell_1$.
\end{cor}
\begin{proof}
Taking $|a_n|=1$ in Definition~\ref{defWeightedWeakSLC}
gives the first implication. If $\phi(1)\le2$, the converse
follows from Theorems~\ref{M2} and
\ref{MainWeightedCharacterization}. If $\phi(1)\ge3$,
Theorem~\ref{weakNotAGExample} supplies a basis satisfying
\weakphiSLCone{} but not \phiAGone{}, and hence not
\weightedweakphiSLCone{}.
\end{proof}

\appendix
\section{Equivalent descriptions of the enlargement functions}\label{phi_equiv}
\begin{proof}[Proof of Proposition~\ref{propPhiEquivalent}]
Suppose (i) holds. Since $\phi(0)=0$ and $\phi$ is strictly
increasing and integer-valued, $\phi(m)\ge m$ for all $m\ge0$.
Moreover, if $n\ge k$ and $\phi(n)\le m$, then strict monotonicity
at integer arguments yields
\[
 \phi(n-k)\ \le\ \phi(n)-k\ \le\  m-k,
\]
which is (ii).

Conversely, assume (ii), and set $\gamma(j)=\phi(j)-j$ for $j\ge0$.
It follows that $\gamma(j)\ge0$ and $\gamma(0)=0$. Given integers
$0\le a<b$, apply (ii) with $n=b$, $k=b-a$, and $m=\phi(b)$ to get
\[
 \phi(a)\ \le\ \phi(b)-(b-a),
 \qquad\text{hence}\qquad
 \gamma(a)\ \le\ \gamma(b).
\]
Therefore (iii) holds.

Finally, if (iii) holds, then for every $j\ge0$,
\[
 \phi(j+1)-\phi(j)
 \ =\ 1+\gamma(j+1)-\gamma(j)\ \ge\ 1.
\]
Thus $\phi$ is strictly increasing, and $\phi(0)=0$, proving (i).
\end{proof}

\section{Approximations by finitely supported vectors}
The following elementary finite-support reduction is used in several proofs.
\begin{lem}\label{FiniteSupportReduction}
Suppose $F\subset\mathbb N$ is finite, $x\in X$, and
$F\cap\supp(x)=\emptyset$. Then there are finitely supported vectors
$x_r$ with $F\cap\supp(x_r)=\emptyset$ such that
\[
 \|x_r-x\|\longrightarrow0
 \quad\text{and}\quad
 \|x_r-x\|_\infty\longrightarrow0.
\]
\end{lem}

\begin{proof}
Since $\overline{\operatorname{span}}\{e_n:n\in\mathbb N\}=X$, there exists a sequence of finitely supported
vectors $(z_r)$ such that $z_r\to x$ in norm.

For each $r\in\mathbb N$, define $x_r = z_r-P_F(z_r)$.
Clearly, each $x_r$ is finitely supported, and $F\cap\operatorname{supp}(x_r)=\emptyset$
for every $r\in\mathbb N$.

Next, we show that $x_r\to x$ in norm.
Since $F$ is finite and each coordinate functional
$e_n^*$ is continuous, for every $y\in X$, we have
\begin{align*}
\|P_F(y)\|
&\ \le\ \sum_{n\in F}|e_n^*(y)|\|e_n\|\\
&\ \le\ \left(\sum_{n\in F}
       \|e_n^*\|\|e_n\|\right)\|y\|.
\end{align*}
Thus, $P_F$ is a bounded linear operator.

Since $F\cap\operatorname{supp}(x)=\emptyset$,
we have $P_F(x)=0$. Consequently,
\begin{align*}
\|x_r-x\|
&\ =\ \|z_r-P_F(z_r)-x\|\\
&\ =\ \|(z_r-x)-P_F(z_r-x)\|\\
&\ \le\ (1+\|P_F\|)\|z_r-x\|
\ \longrightarrow\ 0.
\end{align*}

Finally, since the coordinate functionals are
uniformly bounded, we have
\[
M\ :=\ \sup_{n\in\mathbb N}\|e_n^*\|\ <\ \infty.
\]
It follows that
$$\|x_r-x\|_\infty\ =\ \sup_{n\in\mathbb N}|e_n^*(x_r-x)|\ \le\ M\|x_r-x\|
\ \longrightarrow\  0.$$
This completes the proof.
\end{proof}

\end{document}